%% file: main.tex
\documentclass{article}
\newcommand*{\inserttitle}{Small lattice polytopes have few vertices}
\usepackage{marvosym}

\input{header}

\begin{document}
\thispagestyle{plain}

\null
\vspace{0.5\baselineskip}
\begin{center}
    {\LARGE
    \scshape\bfseries
    small lattice polytopes have few vertices}\\
    \Large A proof of Andrew's theorem\\[0.5em]
    Travis Dillon
\end{center}
\vspace{1\baselineskip}

The purpose of this paper is to prove that a lattice polytope (that is, a polytope whose vertices are lattice points) with many vertices must also have large volume (or, the other way around, a lattice polytope with small volume does not have many vertices).

\begin{theorem}\label{thm:volume-vertex}
    For any lattice polytope $P$ in $\R^d$,
    \[
        |\operatorname{vert}(P)| \leq c_d \Vol(P)^{\frac{d-1}{d+1}}.
    \]
    The constant $c_d$ depends only on the dimension $d$.
\end{theorem}
\noindent
In comparison, $|\Z^d \cap P| \leq c_d' \Vol(P)$ is the best possible inequality comparing the volume of a lattice polytope to the number of lattice points it \emph{contains} (see \cref{thm:volume-lattice_point}).

As Imre B\'ar\'any's notes in his paper \textit{Random points and lattice points in convex bodies} \cite{barany-random_lattice}, the original proof of G.\,E. Andrews in 1963 is ``not easy'' and while ``there are several other proofs available,\textdots none of them is simple\slimdot''\footnote{See the citations after Theorem 13.1 in \cite{barany-random_lattice} for references to other proofs.} That's a shame, because this is a very nice theorem: It's tight up to a constant (see \cref{thm:volume-vertex-tight}), and it shows a fundamental difference between the number of lattice points \emph{contained in} a lattice polytope and the number of \emph{vertices} of a lattice polytope.

This note is a retelling of Konyagin and Sevast'yanov's proof from \cite{KS-volume_vertex}. My intention is to make their strategy more evident and foreground the underlying geometry. Reading their proof is like being grabbed firmly by the arms and marched steadily and unyieldingly through a formidable concrete hallway until you stumble over the conclusion you were trying to reach. It's there, sure enough, but you have to wonder why you took such an odd journey to get there, and why in the middle of the march you were told to scratch your toes, and, now that you think about it, why you walked backwards rather than forwards. That said, if you dismiss the guide, re-plot the journey, and post a few guiding signs, it's an appealing and satisfying journey.

Let's get to it.

\section{The strategy}

In broadest terms, the strategy is simple: Induction on the dimension.

The case $d=1$ is easy, since any polytope in $\R^1$ has at most two vertices. Now we move up.

Suppose $P\subseteq \R^d$ has $m$ vertices and $k$ facets, which we'll call $F_1,\dots,F_k$; suppose that $F_i$ has $m_i$ vertices. The hyperplane containing $F_i$ intersects $\Z^d$ in an affine sublattice of dimension $d-1$; let $\Pi_i$ be its fundamental parallelotope. \cref{thm:volume-vertex} applied to $F_i$ in this sublattice tells us that
\begin{equation}\label{eq:induction-ineq}
    m_i \leq c_{d-1} \Big(\frac{\Vol_{d-1}(F_i)}{\Vol_{d-1}(\Pi_i)}\Big)^{\frac{d-2}{d}}.\tag{$\ast$}
\end{equation}

To use induction, we bound $m$ by $\sum_{i=1}^k m_i$ and then use the previous inequality to bound each term in this sum by some function of the $(d-1)$-dimensional volume of the facets. But then we can't make the next step: it's impossible to upper bound the surface measure by \emph{any} function of the volume. As an example, take the usual unit hypercube $\conv(\{0,1\}^d)$ and shear the top facet far away from the bottom facet. (For example, so that the vertices are $\{0,1\}^{d-1}\times \{0\}$ and $\{a,a+1\} \times \{0,1\}^{d-2}\times \{1\}$ for some $a \in \Z$.) The volume of this polytope is always 1, but the surface measure tends toward infinity.

Luckily, there is a fix, called the \emph{reverse isoperimetric inequality}: (see \cref{sec:reverse-isoperimetric} for a proof\,\footnote{To focus on the main ideas of the proof, proofs of various lemmas and tangential statements have been moved to \cref{sec:other-proofs}.})
\begin{lemma}
    For any convex body $C$ in $\R^d$, there is a volume-preserving linear transformation $A$ so that
    \[
        \Vol_{d-1}\!\big(\partial A(C)\big) \leq c_d \Vol\!\big(A(C)\big)^{\frac{d-1}{d}}.
    \]
\end{lemma}

Our new strategy is to use the bound
\[
    m_i \leq c_{d-1} \Big(\frac{\Vol_{d-1}\!\big(A(F_i)\big)}{\Vol_{d-1}\!\big(A(\Pi_i)\big)}\Big)^{\frac{d-2}{d}},
\]
which follows from \eqref{eq:induction-ineq} simply because the \emph{ratio} of two volumes is invariant under linear transformation (even though the actual quantity $\Vol_{d-1}\!\big(A(F_i)\big)$ is likely different). The outline of the proof is then:
\begin{enumerate}
    \item Calculate a lower bound on $\Vol_{d-1}\!\big(A(\Pi_i)\big)$ to obtain an upper bound for $m_i$ in terms of $\Vol_{d-1}\!\big(A(F_i)\big)$.
    \item Use the reverse isoperimetric inequality to convert this into a bound on volume.
\end{enumerate}
That's it! Now it's time to follow through.

\section{The calculations}

\subsection{Volume bound on fundamental parallelotopes}

Instead of estimating $\Vol_{d-1}\!\big(A(\Pi_i)\big)$, let's start with the easier task of estimating $\Vol_{d-1}(\Pi_i)$. The strategy is the same for both cases, but the latter is less ornate and therefore easier to understand.

To do find a lower bound for $\Vol_{d-1}(\Pi_i)$, we'll introduce the vector $h_i$, which is a normal vector to $\Pi_i$ of length $|h_i| = \Vol_{d-1}(\Pi_i)$. The excellently convenient fact about this vector is that it has integer coordinates. There are two ways to see this, one by matrix manipulation and the other geometrically; both can be found in \cref{thm:normal-vector}. Either way, now comes the clever bit.

Order the indices so that $|h_1|\leq \cdots \leq |h_k|$ (in other words, so that the volumes of $\Pi_i$ form a nondecreasing sequence). For any $\ell$, we have
\[
    \Vol\!\big(\!\conv(0,h_1,\dots,h_\ell)\big) \leq \Vol(B^d) |h_\ell|^d,
\]
since all the vectors $h_1,\dots,h_\ell$ are contained inside the ball of radius $|h_\ell|$. ($B^d$ is the unit ball.) But we can also get a lower bound for the volume based only on the fact that it contains at least $\ell$ integer points: the points $h_1,\dots,h_\ell$ themselves.

\begin{lemma}\label{thm:volume-lattice_point}
    If $X\subseteq \Z^d$ does not lie in a single hyperplane, then
    \[
        \Vol\!\big(\!\conv(X)\big) \geq \frac{|X|-d}{d!}.
    \]    
\end{lemma}

For a proof, see \cref{sec:volume-lattice_point}. At this point, it would be nice to say that
\[
    \frac{\ell-d}{d!} \leq \Vol(B^d)|h_\ell|^d,
\]
and since $d!$ and $\Vol(B^d)$ are both constants in $d$, we have that $\Vol_{d-1}(\Pi_\ell) = |h_\ell| \gtrsim (\ell-d)^{1/d}$ for every $\ell$. (The symbol $\gtrsim$ means that the inequality is true up to a constant that depends only on the dimension.) But this is only true if $h_1,\dots,h_\ell$ span $\R^d$. \emph{If} this condition holds, combining this lower bound for $\Vol_{d-1}(\Pi_\ell)$ with \eqref{eq:induction-ineq} tells us that
\[
    m_\ell^{d/(d-2)} (\ell-d)^{1/d} \lesssim \Vol_{d-1}(F_i).
\]

Even though this inequality is only true for some $\ell$, we're still able to find an upper bound for $m$:
\begin{lemma}\label{thm:restricted-sum}
    If $t$ is the largest index such that $h_1,\dots,h_t$ is contained in a proper subspace of $\R^d$, then $m \leq \sum_{i=t+1}^k m_i$. 
\end{lemma}
\begin{proof}
    Let $w$ be any vector orthogonal to the $(d-1)$-dimensional subspace containing $h_1,\dots,h_t$. Each vertex of $P$ is contained in a facet whose normal vector is not perpendicular to $w$ (in other words, a facet that is not parallel to $w$), so
    \[
        m
        \leq \sum_{\langle h_i,w\rangle \neq 0} m_i
        \leq \sum_{i=t+1}^k m_i.\qedhere
    \]
\end{proof}
So it will be enough to work with $m_i$ for $i \geq t+1$.

Now we want to transfer this proof scheme to obtain a lower bound for $\Vol_{d-1}\!\big(A(\Pi_i)\big)$. Here's how. Let $A^{-\top}$ denote $(A^{-1})^\top = (A^\top)^{-1}$. The vector $A^{-\top}(h_i)$ is perpendicular to $A(\Pi)$, and $|A^{-\top}h_i| = \Vol_{d-1}\!\big(A(\Pi_i)\big)$. (Verifying this is straightforward linear algebra; see \cref{sec:applying-A}.) Now we can repeat everything from above, but with $A(\Pi_i)$ in place of $\Pi_i$ and $A^{-\top}h_i$ in place of $h_i$. If $r$ is the largest index so that $A^{-\top}h_1,\dots,A^{-\top}h_r$ is contained in a hyperplane, then
\begin{equation}\label{eq:m-restricted-sum}
    m\leq \sum_{i=r+1}^k m_i
\end{equation}
and
\begin{equation}\label{eq:m_i-upper-bound}
    m_i^{d/(d-2)} (i-d)^{1/d} \lesssim \Vol_{d-1}\!\big(A(F_i)\big)
\end{equation}
for every $i \geq r+1$.

This are the inequalities we'll use in the following section.

\subsection{Stringing inequalities}

To start, apply H\"older's inequality to \eqref{eq:m-restricted-sum} with $p=d/(d-2)$:
\begin{equation}\label{eq:holder}
    m
    \leq \sum_{i=r+1}^k m_i
    \leq \Big(\sum_{i=r+1}^k m_i^{\frac{d}{d-2}} (i-d)^{1/d}\Big)^{\frac{d-2}{d}}\ \Big(\sum_{i=r+1}^k (i-d)^{-\frac{d-2}{2d}}\Big)^{2/d}.
\end{equation}
The reason for this peculiar choice is that we found an upper bound for the terms of the first sum in the previous section. So now we just need to tackle the second sum. Let $\hat m := \sum_{i=r+1}^k m_i$. Since $m_i \geq 1$ for each $i$, we have that $k \leq \hat m$. Therefore
\[
    \sum_{i=r+1}^k (i-d)^{-\frac{d-2}{2d}}
    \lesssim \int_{r+1}^k (x-d)^{-\frac{d-2}{2d}}\, dx
    < \frac{2d}{d+2} (k-d)^{\frac{d+2}{2d}}
    \lesssim k^{\frac{d+2}{2d}}
    \leq {\hat m}^{\frac{d+2}{2d}}.
\]

Substituting into \eqref{eq:holder}, we get
\[
    \hat m \leq \Big(\sum_{i=r+1}^k \Vol_{d-1}\!\big(A(F_i)\big)\Big)^{\frac{d-2}{d}} {\hat m}^{(d+2)/d^2}.
\]
Combining the powers of $\hat m$ and bounding the sum by $\Vol_{d-1}\!\big(\partial A(P)\big)$, then applying the reverse isoperimetric inequality, we get
\[
    {\hat m}^{(d^2-d-2)/d}
    \leq \Vol_{d-1}\!\big(\partial A(P)\big)^{\frac{d-2}{d}}
    \lesssim \Vol_d
    !\big(A(P)\big)^{\frac{(d-1)(d-2)}{d^2}}.
\]
Now take each side to the power of $d^2/(d+1)(d-2)$:
\[
    m
    \leq \hat m
    \lesssim \Vol_d\!\big(A(P)\big)^{\frac{d-1}{d+1}}
    = \Vol_d\!\big(P\big)^{\frac{d-1}{d+1}},
\]
since $A$ is a volume-preserving transformation.\qed

\section{A stronger result}

With a small addition to the proof, we can obtain something stronger. A \emph{tower} or \emph{flag} of a polytope $P$ is a sequence $G_0\subset G_1 \subset \cdots \subset G_d$ where each $G_k$ is a $k$-dimensional face of $P$. We let $T(P)$ denote the total number of towers in $P$.

\begin{theorem}\label{thm:volume-tower}
    For any lattice polytope $P$ in $\R^d$,
    \[
        T(P) \lesssim \Vol(P)^{\frac{d-1}{d+1}}.
    \]
\end{theorem}

The towers satisfy the recurrence $T(P) = \sum_{i=1}^k T(F_i)$, which means that you can nearly prove \cref{thm:volume-tower} by copying the proof of \cref{thm:volume-vertex}, replacing $m$ by $T(P)$ and $m_i$ by $T(F_i)$. The only part that falters is \cref{thm:restricted-sum}. It's simply never true that $T(P) \leq \sum_{i=r+1}^k T(F_i)$. However, throw in a constant and everything is fine:

\begin{lemma}
    If $t$ is the largest index such that $h_1,\dots,h_t$ is contained in a proper subspace of $\R^d$, then
    $
        \sum_{i=t+1}^k T(F_i) \geq \frac{1}{d} T(P).
    $
\end{lemma}
\begin{proof}
    As before, let $w$ be any vector orthogonal to the $(d-1)$-dimensional subspace containing $h_1,\dots,h_t$. We will show that
    \[
        \sum_{\langle h_i,w\rangle \neq 0} T(F_i) \geq \frac{1}{d} T(P).
    \]
    The left-hand sum is equal to the number of towers of $P$ that do not include a face that is parallel to $w$. (A face is \emph{parallel} to $w$ if its affine span contains a translate of $w$.) Given a vertex $v$, let $\mathcal T_v(P)$ denote the set of towers of $P$ whose 0-dimensional face is $v$, and let $\mathcal T_v^w(P)$ denote the subset of towers in $T_v^w(P)$ which do not include a face parallel to $w$. We will prove that $|\mathcal T_v^w(P)|\geq \frac{1}{d} |\mathcal T_v(P)|$. Since $T(P) = \sum_v |\mathcal T_v(P)|$, that suffices to prove the lemma.

    If $G_k$ is a $k$-dimensional face not parallel to $w$, then it is contained in at least $d-k$ faces of dimension $k+1$ and at most one of them is parallel to $w$. Using this fact starting from a single vertex, we have
    \[
        |\mathcal T_v^w(P)|
        \geq \frac{d-1}{d} \frac{d-2}{d-1} \cdots \frac{2}{3} \frac{1}{2} |\mathcal T_v(P)|
        = \frac{1}{d} |\mathcal T_v(P)|.\qedhere
    \]
\end{proof}

\cref{thm:volume-tower} also implies a version of \cref{thm:volume-vertex} for faces of all dimension:
\begin{theorem}
    There exists a constant $c_d$ so that for every lattice polytope $P$ in $\R^d$, the number of $k$-dimensional faces of $P$ is at most $c_d\Vol(P)^{\frac{d-1}{d+1}}$.
\end{theorem}


\vspace{4\baselineskip}
\appendix
\section{The rest of the proofs}\label{sec:other-proofs}

\subsection{Reverse isoperimetric inequality}\label{sec:reverse-isoperimetric}
John's theorem quickly proves this theorem, though not with an optimal constant.

Choose $A$ so that the largest-volume ellipsoid contained inside $A(C)$ is a ball $B$. Since both sides of the inequality are linear under scaling by a constant factor, we may assume that $\Vol\!\big(A(C)\big) = 1$. It suffices therefore to show that $\Vol_{d-1}\!\big(\partial A(C)\big)$ is bounded by a constant. This follows from John's theorem, which implies that $A(C) \subseteq d\cdot B \subseteq d\cdot B^d$, where $B^d$ is the unit ball. So we can take $c_d = d^{d-1} \Vol_{d-1}(\partial B^d)$.

\subsection{\texorpdfstring{\cref{thm:volume-vertex}}{The theorem} is tight}\label{thm:volume-vertex-tight}
The example is surprisingly simple. Let $X$ be the set of points
\[
    X = \big\{ (x, \|x\|^2) \in \Z^d \,:\, x \in \{-n,-n+1,\dots,n-1,n\}^{d-1} \big\}.
\]
We will take $P = \conv(X)$.

Since $x \mapsto \|x\|^2$ is a convex function, $X$ is a set of points in convex position, so $P = \conv(X)$ has exactly $|X| = (2n+1)^{d-1}$ vertices. The exact volume of $P$ might be difficult to calculate, but $P$ certainly contains the pyramid with apex at the origin and whose base has vertices $\{\pm n\}^{d-1} \times \{dn^2\}$. The volume of this cone is
\[
    \frac{1}{d} (2n)^{d-1} dn^2 = 2^{d-1} n^{d+1}.
\]
Thus
\[
    \Vol(P)^{\frac{d-1}{d+1}}
    \leq C_d\, \lvert\operatorname{vert}(P)\rvert
\]
for a suitable choice of constant $C_d$.

\subsection{Normal vector to the parallelotope}\label{thm:normal-vector}

\begin{proposition}
    Given $d-1$ vectors $v_1,\dots,v_{d-1} \in \Z^d$, a vector orthogonal to their span and whose length is equal to the $(d-1)$-dimensional volume of the parallelotope generated by $v_1,\dots,v_{d-1}$ is a member of $\Z^d$.
\end{proposition}
\noindent
We prove this in two ways.

\subsubsection*{The matrix method}
Arrange the vectors $v_1,\dots,v_{d-1}$ as columns in a matrix $M$, and define $h$ by its coordinates: $h_i$ is the determinant of $M$ after deleting the $i$th row. This is a general way of producing a vector orthogonal to $d-1$ others. To verify that $h$ is orthogonal to $v_i$, we can use the cofactor formula for the determinant to see that $\langle h, v_i\rangle$ is equal to the determinant of the matrix with columns $v_1,\dots,v_{d-1}, v_i$, which is 0, since 1 column is repeated.

If $v_1,\dots,v_{d-1} \in \Z^d$, then by definition $h \in \Z^d$. We need to calculate $|h|$. Let $B$ be the $(d-1)$-dimensional volume of the parallelotope generated by $v_1,\dots,v_{d-1}$. On the one hand, $|h|^2 = h_1^2 + \cdots + h_d^2$. Also, since $h$ is orthogonal to $v_1,\dots,v_{d-1}$, we know that $\det(h,v_1,\dots,v_{d-1}) = |h|\cdot B$. On the other hand, direct calculation using the cofactor formula shows that $\det(h,v_1,\dots,v_{d-1}) = h_1^2 + \cdots + h_d^2$. We conclude that $|h|^2 = |h|\cdot B$, so $|h| = B$, as desired.

\subsubsection*{The geometric method}

If $v \in \Z^d$ and the set of coordinates of $v$ has greatest common divisor 1 (in which case $v$ is called \emph{primitive}), then the projection of $\Z^d$ onto $\operatorname{span}(v)$ is the set of points $\frac{k}{|v|} v$ with $k \in \Z$. (This is because the vector $w$ projects onto the point $\frac{\langle w,v\rangle}{|v|}v$, and if $v$ is primitive, then there is a solution to $\langle w,v\rangle = k$ for every $k \in \Z$.) If $v \notin \Z^d$, then the projection of $\Z^d$ onto $\operatorname{span}(v)$ is dense.

Let $h'$ be a unit vector orthogonal to $v_1,\dots,v_{d-1}$. The projection of a lattice onto the orthogonal complement of any sublattice is also a lattice, so the projection of $\Z^d$ onto $\operatorname{span}(h')$ is a sublattice. Let $B$ be the $(d-1)$-dimensional volume of the parallelotope generated by $v_1,\dots,v_{d-1}$, let $B_f$ be the volume of a fundamental parallelotope in $\operatorname{span}(v_1,\dots,v_{d-1})$, and let $w$ be a vector that projects to a minimal nonzero vector $\rho(w)$ in $\operatorname{span}(h')$. Then $v_1,\dots,v_{d-1},w$ generate a fundamental parallelotope in $\Z^d$. The volume of this parallelotope is therefore 1, which means that the orthogonal component of $w$ is $1/B_f$. In other words, $|\rho(w)|=1/B_f$.

Therefore the projection of $\Z^d$ onto $\operatorname{span}(h')$ is $\frac{k}{B_f}h'$. Using the reasoning from the first paragraph, we conclude that $B_f \cdot h'$ is a lattice vector. Since $B_f$ divides $B$, we also have that $h := B\cdot h'$ is a lattice vector.

\subsection{Volume vs. total lattice points}\label{sec:volume-lattice_point}

Here is the fundamental fact that connects volume to lattice points:

\begin{lemma}\label{thm:volume-of-simplex}
    The volume of a lattice simplex in $\R^d$ is at least $1/d!$.
\end{lemma}
\begin{proof}
    By elementary calculus, the volume of a cone in $\R^d$ with base $(d-1)$-volume $B$ and height $h$ is $hB/d$. Suppose that $S$ is a simplex, and let $v_1,\dots,v_d$ be the vectors corresponding to the edges emanating from one of its vertices. Induction shows that the volume of $S$ is exactly equal to $1/d!$ times the volume of the parallelotope determined by $v_1,\dots,v_d$. That volume is given by $\lvert\det(v_1,\dots,v_d)\rvert$. If $S$ is a lattice simplex, then $v_1,\dots,v_d$ are integer vectors. Since $\det(v_1,\dots,v_d) \neq 0$, we have
    \[
        \Vol(S)
        \geq \frac{\lvert\det(v_1,\dots,v_d)\rvert}{d!}
        \geq \frac{1}{d!}.
        \qedhere
    \]
\end{proof}

One way to prove \cref{thm:volume-lattice_point} is to actually prove the stronger statement that any lattice polytope can be decomposed into at least $|X|-d$ lattice simplices. \cref{thm:volume-of-simplex} then implies \cref{thm:volume-lattice_point}. The proof of this stronger statement goes by induction on $|X|$. Intuitively, it's pretty clear: remove one vertex $v$ from $X$ and find a lattice simplex in $\conv(X)\setminus \conv(X\setminus v)$; then repeat. Writing it out in full is somewhat tedious, so I'll leave it for you to think about.

There are many examples that show that \cref{thm:volume-lattice_point} is tight (up to the constant). One simple example is $X = \{1,2,\dots,n\}^d$, which has $n^d$ lattice points and a convex hull with volume $(n+1)^d$.

\subsection{Miscellaneous linear algebra}\label{sec:applying-A}

\begin{lemma}
    If $u,v \in \R^d$ are orthogonal and $A$ is any linear transformation, then $A^{-\top}v$ is orthogonal to $Au$.
\end{lemma}
\begin{proof}
    A calculation: $\langle A^{-\top}v, Au\rangle = \langle v, A^{-1}Au\rangle = \langle v,u\rangle = 0$.
\end{proof}

\begin{lemma}
    Given $v_1,\dots,v_{d-1} \in \R^d$, let $Q$ be the parallelotope they generate, and let $h$ be a vector orthogonal to $v_1,\dots,v_{d-1}$ with length $|h|=\Vol_{d-1}(Q)$. If $A$ is a volume-preserving linear transformation, then $A^{-\top}h$ is orthogonal to $A(Q)$ and has length $|A^{-\top}h| = \Vol_{d-1}\!\big( A(Q)\big)$.
\end{lemma}
\begin{proof}
    Let $R$ be the parallelotope generated by $v_1,\dots,v_{d-1},h$. Since $A$ is volume-preserving and $h$ is orthogonal to $Q$,
    \[
        \Vol\!\big(A(R)\big) = \Vol(R) = \Vol_{d-1}(Q)^2.
    \]
    On the other hand, the volume of $A(R)$ is the volume of the base times the length of the orthogonal component of $Ah$:
    \[
        \Vol\!\big(A(R)\big)
        = \frac{\langle Ah, A^{-\top}h\rangle}{|A^{-\top}h|} \Vol_{d-1}(Q)
        = \frac{|h|^2}{|A^{-\top}h|} \Vol_{d-1}(Q)
        = \frac{\Vol_{d-1}(Q)^3}{|A^{-\top}h|}.
    \]
    Chain the inequalities and cancel terms to get $|A^{-\top}h| = \Vol_{d-1}(Q)$.
\end{proof}

\bibliographystyle{amsplain-nodash}
\bibliography{bibliography}

\vspace{1.5\baselineskip}

\noindent
{\small \textsc{Travis Dillon}}\\
{\small \textsc{Department of Mathematics, Massachusetts Institute of Technology, Cambridge, MA, USA}}\\
\textit{email:} \texttt{travis.dillon@mit.edu}

\end{document}

%% file: header.tex
\usepackage[T1]{fontenc}

\usepackage[letterpaper, portrait, left=1in, right=1in, top=0.9in, bottom=1in, footnotesep=1.5\baselineskip]{geometry}
\usepackage[dvipsnames]{xcolor}
\usepackage{contour}
    \contourlength{0.5pt}
\usepackage{tocloft}

\usepackage{environ} 
\usepackage[framemethod=tikz]{mdframed}
\usepackage{wrapfig}
\usepackage{mathtools}
    \usetikzlibrary{calc, math, cd, arrows.meta, braids, decorations.pathreplacing, decorations.markings, decorations.pathmorphing, bending, shapes}
    \usepackage{tikz-3dplot}
    \tikzset{vertex/.style={draw, shape=circle, inner sep=1.5pt, minimum size=4pt}}
    \tikzset{<->/.tip={Latex}}
    \tikzset{shorten > = 2pt, shorten <=2pt}
    \tikzset{smallnode/.style={every node/.style={draw, fill=black, shape=circle, inner sep=0pt, minimum size=4pt}, scale=0.7}}
    \tikzset{drawnode/.style={fill,shape=circle,inner sep=0pt, minimum size=3pt}}
\usepackage{pgfornament}
\usepackage{fourier-orns} 
\usepackage{stmaryrd}
\usepackage{halloweenmath}
\usepackage[shortlabels]{enumitem}
    \newlength{\circlabelwidth}
    \setlist{nosep}
    \setlist[enumerate]{label=\textup{\arabic*.}}
    \newlist{subprob}{enumerate}{2}
        \setlist[subprob,1]{label={(\roman*)}}
        \setlist[subprob,2]{label={(\arabic*)}}
    \setlist[itemize]{labelindent=10pt,labelwidth=\circlabelwidth,leftmargin=!,label=$\circ$}
    \newlist{problems}{enumerate}{3}
        \setlist[problems,1]{before=\setupstar,label=\textup{\arabic*.}, itemsep=2pt, topsep=8pt,ref=\textup{\arabic*}}
        \setlist[problems,2]{before=\setupstar,label=(\alph*),parsep=0pt}
        \setlist[problems,3]{before=\setupstar,label=(\roman*),parsep=0pt}

\usepackage{fancyhdr}
\usepackage{dopestyle}

\makeatletter
    \renewcommand\@makefntext[1]{\leftskip=0em\hskip-0em\@makefnmark\,#1}
\makeatother

\surroundwithmdframed[skipabove=0.5\baselineskip, skipbelow=0.5\baselineskip, leftmargin=3pt, rightmargin=3pt, innerleftmargin=7pt, innerrightmargin=7pt, roundcorner=10pt, linewidth=2pt, linecolor=red!40, backgroundcolor=red!5]{headsup}

\surroundwithmdframed[topline=false, bottomline=false, innertopmargin=0pt, innerbottommargin=0pt, innerleftmargin=2pt, innerrightmargin=2pt, linewidth=0.2mm]{quote}

\surroundwithmdframed[topline=false, bottomline=false, innertopmargin=2.5pt, innerbottommargin=2.5pt, innerleftmargin=-10pt, leftmargin=-10pt, innerrightmargin=-10pt, rightmargin=7.7pt, linewidth=0.4mm]{quote}

\NewEnviron{method}{
    \parbox{\textwidth}{
        \textbf{\textsc{Method}}
        \begin{mdframed}[innerleftmargin=4pt,innerrightmargin=4pt,skipabove=3pt,skipbelow=0pt]
            \BODY
        \end{mdframed}
    }
}

\theoremstyle{itcaps}
\newtheorem{theorem}{Theorem}
\newtheorem*{theorem*}{Theorem}

\newtheorem{lemma}{Lemma}
\newtheorem{proposition}[lemma]{Proposition}

\theoremstyle{solved}

\theoremstyle{caps}

\newtheorem{exampleprimitive}[theorem]{Example}

    \crefname{exercise}{Exercise}{Exercises}

\theoremstyle{remark}

\numberwithin{equation}{section}

\newcommand*{\newword}[2][]{\emph{#2}\index{%
    \ifx&#1&%
       #2%
    \else%
       #1%
    \fi}%
} 
\newcommand*{\oldword}[2][]{#2\index{%
    \ifx&#1&%
       #2%
    \else%
       #1%
    \fi}%
} 

\DeclareMathOperator{\Vol}{Vol}

\DeclareMathOperator{\conv}{conv}

\let\phi\varphi

\let\epsilon\varepsilon

\let\oldchi\chi
\renewcommand{\chi}{\raisebox{1pt}{$\oldchi$}}

\title{\inserttitle}
\date{}

\allowdisplaybreaks
